\documentclass[10pt]{amsart}
\usepackage{latexsym,amsmath,amssymb,graphicx,epsfig,color,hyperref,supertabular,cite}
\usepackage{pdfpages}
\usepackage{pst-node}
\usepackage{tikz-cd} 
\usepackage[all,cmtip]{xy}
\usepackage[british]{babel}
\usepackage[T1]{fontenc}
\usepackage{enumitem}   
\usepackage{amsmath,tikz}
\DeclareMathAlphabet{\mathpzc}{OT1}{pzc}{m}{it}
\DeclareMathAlphabet{\mathcalligra}{T1}{calligra}{m}{n}
\def\pplogo{\vbox{\kern-\headheight\kern -15pt
		\halign{##&##\hfil\cr&{%\sc
				\ppnumber}\cr\rule{0pt}{2.5ex}&\ppdate\cr} }} \makeatletter
\def\ps@firstpage{\ps@empty \def\@oddhead{\hss\pplogo}%
	\let\@evenhead\@oddhead % in case an article starts on a left-hand pagedix
	
}
\renewcommand{\thetable}{\@Alph\c@table}
\makeatother

\DeclareFontFamily{OT1}{rsfs10}{}
\DeclareFontShape{OT1}{rsfs10}{m}{n}{ <-> rsfs10 }{}
\DeclareMathAlphabet{\mathscript}{OT1}{rsfs10}{m}{n}

\suppressfloats
\def\ppnumber{\vbox{\baselineskip16pt }}
\def\ppdate{}
\date{} %necessary, to disable automatic dating by LaTeX

\newsavebox{\fmbox}

\author{Antonella Grassi and Timo Weigand}

\title{Elliptic  threefolds with high Mordell-Weil rank}

\address{Antonella Grassi, Dipartimento di Matematica, Universit\`a di Bologna, Bologna, Italy}
\address{Department of Mathematics, University of Pennsylvania, Philadelphia, PA 19104}
\email{antonella.grassi3@unibo.it}
\address{Timo Weigand, II. Institut f\"ur Theoretische Physik, Universit\"at Hamburg, Luruper Chaussee 149, 22607 Hamburg, Germany}
\address{Zentrum f{\"u}r Mathematische Physik, Universit\"at Hamburg, Bundesstrasse 55, 20146 Hamburg, Germany}
\email{timo.weigand@desy.de}

\theoremstyle{plain}
\newtheorem{thm}{Theorem}

\newtheorem{corollary}[thm]{Corollary}

\newtheorem{proposition}[thm]{Proposition}
\newtheorem{result}[thm]{Result (Physics)}

\newtheorem{lemma}[thm]{Lemma}

\theoremstyle{definition}
\newtheorem{definition}[thm]{Definition}

\theoremstyle{remark}
\newtheorem{remark}[thm]{Remark}

\numberwithin{thm}{section}

\def\C{\mathbb{C}}
\def\Z{\mathbb{Z}}

\def\Q{\mathbb{Q}}

\def\P{\mathbb{P}}

\def\co{\mathcal O}

\newcommand{\D}{{\mathfrak d}}
\newcommand{\E}{{\mathfrak e}}

\def\Sp{\operatorname{PSp}}

\def\MW{\operatorname{MW}}
\def\ST{\operatorname{ST}}
\def\NDE{\operatorname{NDE}}

\def\rk{\operatorname{rk}}

\def\mw{\operatorname{MW}}

\def\u1{\operatorname{U(1)}}
\def\NE{\operatorname{NE}}

\def\F{f_{\Sigma'}}
\def\pu{{\mathbb P^1}}
\def\pd{{\mathbb P^2}}
\def\p1{{\pi_i}}
\def\Do{{D_i^i}}
\def\Du{{D_{i+1}^i}}
\def\Dt{{D_{i+4}^i}}
\def\D3{{D_{i+3}^i}}
\def\F{{\p1^*(f)}}

\def\E{{\mathcal E}}
\def\PA{\mathcal{P}_{j}^{i,A}}
\def\PB{\mathcal{P}_{j}^{i,B}}
\def\sk{{\mathsf{s}_k}}
\def\soh{{\hat{\mathsf{s}}_0}}
\def\skh{{\hat{\mathsf{s}}_k}}

\def\lkh{{\hat{\mathsf{\ell}}_l}}
\def\fh{{\hat{\mathsf{f}}}}

\newcommand{\be}{\begin{equation}}
\newcommand{\ee}{\end{equation}}
\newcommand{\bea}{\begin{eqnarray}}
\newcommand{\eea}{\end{eqnarray}}

\newcommand{\sfs}{\mathsf{s}}

\begin{document}
				\date{\today}
		
		\ppnumber{\qquad \quad \quad \quad \quad \qquad \quad \quad \quad \quad\qquad \quad \quad \quad \quad  \qquad \quad \quad \quad \quad \qquad \quad \quad \quad \quad ZMP-HH/21-9\\ }

	\date{\today}
	

\begin{abstract} 
 We present the first examples of smooth elliptic Calabi-Yau  threefolds with Mordell-Weil rank 10, the highest currently known value.
They are given by the Schoen threefolds introduced by Namikawa; there are six isolated fibers of Kodaira Type IV. 
We explicitly compute the Shioda homomorphism %for the %generators of the
% Mordell-Weil group 
and 
 %their
 the induced height pairing.  Compactification of F-theory on these threefolds gives an effective theory in six dimensions which contains  ten abelian gauge group factors.
 We compute the massless matter spectrum. %, i.e. the charges of the massless hypermultiplets and their multiplicities. 
In particular, we  show that the charged singlet matter need not reside at enhancement loci of Type $I_2$, as previously believed. 
 We relate the multiplicities of the massless spectrum to   genus-zero Gopakumar-Vafa  invariants and other geometric  quantities of the Calabi-Yau.
We show that the gravitational and abelian anomaly cancellation conditions are satisfied. We prove a Geometric Anomaly Cancellation equation and we deduce birational equivalence for the quantities in the spectrum.
We explicitly describe a Weierstrass model over $\pd$  of the
 Calabi-Yau threefolds as a log canonical model and compare  it to 
a construction by Elkies and classical results of Burkhardt.

\end{abstract}

\maketitle
\section{Introduction}

The Mordell-Weil group of sections of an elliptically fibered Calabi-Yau variety  is of 
considerable interest also in physics:  it has a special role in establishing an upper bound on the number of massless particle species in a consistent theory of quantum gravity. In fact, the rank of the Mordell-Weil group of an elliptically fibered Calabi-Yau threefold $X$ determines  the rank of the abelian (non-Cartan) gauge algebra  in compactifications of F-theory (see for example \cite{Weigand:2018rez,Cvetic:2018bni}). It 
is thereby directly related to aspects of quantum gravity.
Consistency conditions of certain BPS strings in F-theory compactifications \cite{Kim:2019vuc} imply various bounds on the rank of the abelian gauge group in minimally supersymmetric compactifications \cite{Lee:2019skh}. The results of \cite{Lee:2019skh} hence yield interesting implications for algebraic geometry:
 the bound predicted by physics implies that  on an elliptic K3 surface  $0 \leq {\rm rk}(\MW({\rm K3})) \leq 18$ and for  elliptically fiberd Calabi-Yau threefolds $X \to B$,   $ {\rm rk}(\MW(X)) \leq 20$ if  $B \neq \mathbb P^2$ and $ {\rm rk}(\MW(X)) \leq 24$ if $B= \mathbb P^2$ (though it has been conjectured that
both bounds can be sharpened further).
For elliptic K3 surfaces the bounds are   in agreement
with known bounds in mathematics \cite{CoxMW1982} and  all such possible Mordell-Weil ranks 
are explicitly realized  \cite{Kuwata2000MW}, \cite{Kloosterman2007}.
Even for K3, however, it is not feasible to find explicit generators for the Mordell-Weil group for all such cases.

For  elliptic Calabi-Yau threefolds, by contrast, no bound to the rank of the Mordell-Weil group is known in the mathematics literature.
This highly motivates the search for elliptic fibrations for Calabi-Yau threefolds with high Mordell-Weil rank. 
In Section \ref{NRConstruction} we present smooth 
elliptic fibrations $X_i \to B$ with $ {\rm rk}(\MW(X_i)) = 10$, the highest currently known value, and we investigate their properties as elliptic varieties.  The discriminant of the elliptic fibration
is supported on six cuspidal curves on the base $B$, the generic singular fibers are of Kodaira Type $II$ and enhance to 
Kodaira Type $IV$ over
the six cusp points of the cuspidal curves  ({\it Theorems \ref{description}, \ref{MW}}).

The $X_i$ come from  ``the Namikawa examples" \cite{NamikawaStratified, RossiMIchele20174NamikawaEx} studied by Namikawa and Rossi for  their deformation properties. They are resolutions of  threefolds of the form $\bar{X} \stackrel{def}=B \times _{\P^1} B'$, with $B$ and $B'$ certain rational elliptic surfaces. These were first introduced by Schoen in \cite{Schoen87}, and are often referred to as ``the Schoens".  Depending on the type and relative location of the singular fibers of the  two rational elliptic surfaces,  $\bar X$ can be smooth or singular, with singularities of different types. Schoen first studied particular configurations such that $\bar X$ is birational to a smooth Calabi-Yau.  The Schoens have interesting arithmetic properties and they have been studied  also in  many other contexts, from   birational geometry to  string theory. In the particular context of studying the Mordell-Weil rank of Calabi-Yau threefolds, the authors of \cite{Morrison:2016lix}, building on \cite{KapustkaGM}, present
several examples of Schoen varieties with a Mordell-Weil rank of up to $9$. 
 We  conjecture that 
the
Namikawa-like  examples lead to the maximal possible Mordell-Weil rank within the class of Schoen manifolds, as we point out 
before Section \ref{sec_Cohom}.

The geometry of a Calabi-Yau  is closely related to the  massless particle spectrum and the relations that the quantities  in the spectrum  must satisfy, the anomaly cancellation conditions. This connection brings  us to four questions:  1)  to   establish  a dictionary for the correspondence,  2) to find a geometric counterpart for  the ``anomaly cancellation conditions",   3) to calculate explicitly  the geometric quantities in the spectrum and 4) to extract the geometric properties implied by the anomalies.
  In Section \ref{sec_Anomalies} we address 1) and 2): we review the results from physics which provide the   dictionary for the correspondence, as well as  for the anomaly cancellation conditions in subsection \ref{PhysicsResults};  
in  subsection \ref{GeomAnom} we   define a geometric counterpart formula for the gravitational anomaly cancellation condition, the {\it Geometric Anomaly Equation \ref{genconj}} (along the lines of \cite{GrassiWeigand2} where we  write a more general formula).

To  address 3), that is to  evaluate the spectrum for the Namikawa threefolds, 
the gravitational and gauge anomalies,
 and the   Geometric Anomaly Equation  \ref{genconj}, we need to explicitly determine the Poincar\'e pairing between $H^2(X_i, \Z) $  and $H_2(X_i, \Z)$ ({\it Propositions \ref{intersections} and \ref{lDoDo}})), the Shioda map ({\it Corollary \ref{intersectionsForAnom}}), the height pairings ({\it Corollary \ref{heightp}}), the  relative  genus-zero Gopakumar-Vafa  invariants of holomorphic curves ({\it Proposition \ref{GV}}) and other  geometric invariants of the elliptic Calabi-Yau $X_i$ ({\it Corollaries \ref{intersectionsForGrAnom} and \ref{intersectionsForGaugeAnom}}). The computations leading to the spectrum  are involved. 

The results allow us to  compute the spectrum ({\it Property \ref{spectrum}}) and the $\u1$ charges ({\it Proposition \ref{Charges-NR}}). In particular the analysis exemplifies that the charged singlet matter need not  reside at enhancement loci of Type $I_2$, as previously believed. We verify that the anomaly cancellation conditions in physics are satisfied ({\it Proposition \ref{NR-anS}})  by proving the mathematical counterparts of the {gravitational and $\u1$} anomaly equations \cite{Park:2011ji},
along the lines of what was stated in \cite{GrassiWeigand2}.  As a consequence we obtain birational invariants of the non $\Q$-factorial singularities of the Weierstrass model $\bar X \to B$ ({\it Corollary \ref{birinv}}), which answers 4).

In the last Section  \ref{sec_Elkies} we analyse a  family of Weierstrass models  $W_{\rm NDE} \to \P^2$,   constructed by Elkies \cite{Elkies}, with $\rk \MW(W_{\rm NDE}/ \pd)=10$; one particular  model shares  similarities with the Namikawa threefolds.    $W_{\NDE} $ is numerically Calabi-Yau, but  
Elkies does not make any statement about its  minimal resolution. 
 We compare  the Weierstrass models over $\pd$  of the Namikawa-Rossi  threefolds  with  the ones constructed by Elkies  
 by   explicitly describing a Weierstrass model  $W_{\pd} \to \pd$  of the Namikawa Calabi-Yau as a suitable log canonical model ({\it Corollary \ref{Model}, Theorem \ref{lc} and Corollary \ref{Elkies}}).   Then we take the first steps in addressing the question of  whether 
  $W_{\NDE} \to \P^2$
  is  birationally Calabi-Yau,  by building on  classical results of Burkhardt, 
 leaving the construction of a  elliptic Calabi-Yau  with $\rk (\MW)=10$ in \cite{Elkies} conjectural. 

\smallskip

\paragraph{\bf Acknowledgements}
 We thank M. Rossi for correspondence, and M. Liu and  S. Verra for  helpful conversations. We also thank the referees for useful comments.
  The work of A.G. is partially supported by PRIN ``Moduli and Lie Theory''. A.G. is a member of GNSAGA of INDAM.
  The work of T.W. is supported in part by Deutsche Forschungsgemeinschaft under Germany's Excellence Strategy EXC 2121 Quantum Universe 390833306.

\section{The Namikawa-Rossi Construction}\label{NRConstruction}

Let $r: B \to  \P^1$ be a smooth  rational elliptic surface with section and  $6$ cuspidal fibers, that is  $6$ fibers of Kodaira Type $II$.  $B$  is defined by  the Weierstrass equation $y^2z=x^3 + {b} z$  in the projective bundle $\P(\mathcal{E}) =\P( \co_{\P^1}(3) \oplus   \co_{\P^1}(2) \oplus  \co_{\P^1})$, where  $ b \in H^0(\pu, \co_\pu(6)$  is a general section.
Let $r': B' \to  \P^1$ be a different copy of the same surface,  with   Weierstrass equation  $u^2 w=v^3 + {b} w.$ 

 \begin{lemma}\cite{RossiMIchele20174NamikawaEx,NamikawaStratified}
		The threefold  $\bar{X} \stackrel{def}=B \times _{\P^1} B'$ is a Calabi-Yau threefold, singular at $6$ points  $\{P_1, \cdots, P_6 \}$ of local analytic  equation $\bar{\mathrm{x}}^3-\bar{\mathrm{v}}^3-\bar{\mathrm{y}}^2+\bar{\mathrm{u}}^2=0$. 
	\end{lemma}
\begin{lemma}\label{auto}\cite{NamikawaStratified, RossiMIchele20174NamikawaEx}
	The threefold  $\bar{X} \stackrel{def}=B \times _{\P^1} B' \subset \P(\mathcal{E}) \times \P(\mathcal{E}) \times \pu$ is endowed with an automorphism $\tau$ of order $6$ induced by the automorphism of the ambient space: 
	\begin{eqnarray*}
		\tau_{\mathcal E}:  \P(\mathcal{E}) \times \P(\mathcal{E}) \times \pu &\to&  \P(\mathcal{E}) \times \P(\mathcal{E}) \times \pu \\
		( [x,y,z], \ [v,u,w], \ [\lambda_0, \lambda_1]) &\mapsto& ( [ x,y,z], \ [\epsilon v,-u,w], \ [\lambda_0, \lambda_1]),
	\end{eqnarray*}	
	with $\epsilon$  a primitive cube root of the unity.
\end{lemma}
\begin{definition}\label{autodef}
Let $D_i  \stackrel{def}= \bar X \cap \{(\tau_{\mathcal E})^i  ([x,y,z], \ [x,y,z], \ [\lambda_0, \lambda_1]) \}, \ \ 0 \leq i \leq 5 $.  
\end{definition}

$D_0$ is then the diagonal.

\begin{lemma}\label{localequ}\cite{NamikawaStratified, RossiMIchele20174NamikawaEx} Each divisor $D_i$ contains the singular locus $\{P_1, \cdots, P_6\}$ of $\bar X$.  
	\begin{enumerate}
		\item 
		The local equation around a fixed point  $P_j \in \bar X$, $j = 1, \cdots, 6$, can be written as 
		\begin{equation*}
			xv[(1+\epsilon )v-\epsilon x]=yu \,.
		\end{equation*}
		\item  The local equations of $D_i, D_{i+1}, D_{i+3}, D_{i+4}$, with the indices taken mod $6$,  around  $P_j \in \bar X$ can be taken  respectively to be 
		\begin{equation*}\small{
		%	D_i = 
		\left\{    
				\begin{matrix}
					x=0\\
					y=0
				\end{matrix}\right.  , \quad \quad
		%	D_{i+1}= 
		 \left\{    
				\begin{matrix}
					v=0\\
					u=0
				\end{matrix}\right.  , \quad \quad
			%	D_{i+3} = 
			 \left\{    
				\begin{matrix}
					x=0\\
					u=0
				\end{matrix}\right. , \quad \quad
			%D_{i+4} = 
			 \left\{    
				\begin{matrix}
					v=0\\
					y=0
				\end{matrix}\right. .}
		\end{equation*}	
	\end{enumerate}
{\rm [Note a change in  notation with respect to  \cite{RossiMIchele20174NamikawaEx}, in particular  for  $D_{i+4}$.]}
\end{lemma}

\begin{remark} Note in fact that $\forall \, i, \ 0 \leq i \leq 5$, we can write the local equation around a fixed point  $P_j \in \{P_1, \cdots, P_6\}$ of $\bar X$  as
	\begin{equation*}
		(\bar{\mathrm{x}}-\bar{\mathrm{v}}) \cdot (\bar{\mathrm{x}}-\epsilon\bar{ \mathrm{v}}) \cdot (\bar{\mathrm{x}}-\epsilon ^2\bar{\mathrm{v}})  =(\bar{\mathrm{y}}+\bar{\mathrm{u}} ) \cdot (\bar{\mathrm{y}}-\bar{\mathrm{u}} )
	\end{equation*}
	with 
	\begin{equation*}
		\left\{    
		\begin{matrix}
			y=\bar{\mathrm{y}}+ (-1)^{i} \bar{\mathrm{u}}\\
			u = \bar{\mathrm{y}} - (-1)^{i+1} \bar{\mathrm{u}} 
		\end{matrix}\right.  , \quad
		\left\{    
		\begin{matrix}
			x=\bar{\mathrm{x}} - \epsilon  ^i  \bar{ \mathrm{v}}\\
			v= \bar{\mathrm{x}} - {\epsilon}^{i+1} \bar{ \mathrm{v}}
		\end{matrix}\right. \,.
	\end{equation*}
\end{remark}

 \begin{thm}\label{RossiNamikawa}\cite{RossiMIchele20174NamikawaEx,NamikawaStratified}

The threefold  $\bar{X} \stackrel{def}=B \times _{\P^1} B'$ is a Calabi-Yau threefold, singular at $6$ points  $\{P_1, \cdots, P_6 \}$ of local analytic  equation $\bar{\mathrm{x}}^3-\bar{\mathrm{v}}^3-\bar{\mathrm{y}}^2+\bar{\mathrm{u}}^2=0$. 
\begin{enumerate}
	\item  $b_2(\bar X)=19$ and $\rho(\bar X)=19$,  where $\rho(\bar X)$ denotes the rank of the Picard group. 
\item 	The singularities are terminal and not $\Q$-factorial. 
\item  There are  $6$ %independent 
Weil divisors $D_i,  \ 0 \leq i \leq 5 $, defined in Definition \ref{autodef}, which are not Cartier. 
\item  There exist  $6$ different small projective resolutions $\varphi_i: X_i \to \bar X, \ 0 \leq i \leq 5$. Each $X_i$ is a smooth Calabi-Yau threefold.
\item $X_i$ is obtained by the  consecutive blow up of the divisors $D_i$ and then  of the strict transform of $D_{i+1}$. The small resolution can be described using the  local equations in Lemma \ref{localequ}.

	\item   The exceptional loci of any resolution $\varphi_i: X_i \to \bar X $  are six disjoint pairs $\{\mathcal P_j^{i,A}, \mathcal P_j^{i,B}\}$, $ 1 \leq j \leq 6$, of  $\pu$s with normal bundle $\co_\pu(-1)\oplus \co_\pu(-1)$,  intersecting in one point. The threefolds $X_i$ are connected to each other by flops of the exceptional curves. 
		\item $\chi_{top}(X_i)=36$, $ h^{1,1}(X_i)=\rho( X_i)=21$, $ h^{2,1}(X_i)=3$.
\end{enumerate}
\end{thm}

\begin{lemma}\label{Xsing}   Let  
	$ \bar{\pi}:\textsl{}\bar X 
	%=B \times _{\P^1} B'  \ \stackrel{\bar {\pi}}
	\to B  
	%\ \stackrel{r}\to \pu
	$  be the elliptic fibration on the singular threefold induced by the projection on $B$:
	$ \bar X=  B \times _{\P^1} B'  \ \stackrel{\bar {\pi}}\to B  \ \stackrel{r}\to \pu$. 
Let $\{r^{-1}(\mathfrak{p}_j) \}_{  1 \leq j \leq 6 \ } \in B$ denote the  $6$ cuspidal fibers of $r$, $\mathfrak{p}_j \in \pu$. Let $ \{p_j \in r^{-1}(\mathfrak{p}_j) \}_{  1 \leq j \leq 6 \ }$ be the cuspidal points of  these  fibers in $B$. Then:

	\begin{enumerate} 
		
	\item 	 $\bar \pi (P_j)
		= p_j $, i.e. the image of the singular point $P_j \in \bar X$   is  the cuspidal point of %each  
	the singular  fiber $  r^{-1}(\mathfrak{p}_j)$,  $1\leq j \leq 6$.   
		\item  The support of the discriminant locus of  the elliptic fibration $\bar{\pi}$ is the disjoint union of the $6$ cuspidal curves $\{r^{-1}(\mathfrak{p}_j) \}_{ 1 \leq j \leq 6 \ }$.
		\item %$\bar{\pi}^{-1}(q_j)$, 
		All the singular fibers of $\bar{\pi}$, that is the fiber  over the points $q_j \in r^{-1}(\mathfrak{p}_j)$,  are  cuspidal curves (Kodaira type $II$).
		\item The Weil divisors $D_i$ are smooth and are rational sections of the fibration $\bar{\pi}$.	
	%	\item $\rk(\mw(\bar X)/B)=10$.
	\end{enumerate}
\end{lemma}
\begin{proof}
The statements follow from the construction and the  Lemmas \ref{auto} and \ref{localequ}.
\end{proof}
\begin{definition}\label{dik}
 Let  
$  \varphi_i : X_i  \to \bar X$ be  one of the resolutions in Theorem \ref{RossiNamikawa}, $\ 0 \leq i \leq 5.$  
 
For $0 \leq k \leq 5$, $D^i_k$ denotes the strict transform of the divisor $D_k$ by $\varphi_i$.
\end{definition}
\begin{thm}\label{MW}	 Let  
	$ \bar{\pi}:\textsl{}\bar X 
	\to B  
	$  and $X_i$ be as above and let  $ \bar\pi  \circ \varphi_i \stackrel{def }= \pi_i : X_i  \to B$ be  one of the induced elliptic fibrations, 	$\ 0 \leq i \leq 5.$  The 
	elliptic fibration  $\pi_i: X_i  \to B$  has $\rk (\mw(X_i/B))=10$. \end{thm}
\begin{proof}  The statement  follows from  
	Theorem \ref{RossiNamikawa} and from  the Tate-Shioda-Wazir Theorem \cite{Wazir}. 
	The Tate-Shioda-Wazir Theorem in fact states:\\ $\rk (\mw(X_i/B))=  \rho(X_i)-\rho(B) -1=21-10-1 =10$.
\end{proof}
The  elliptic  fibration of the  smooth Calabi-Yau threefolds is described explicitly as follows: 

\begin{thm}\label{description}
	  Let  
	$ \bar{\pi}:\textsl{}\bar X 
	\to B  
	$  and $X_i$ be as above and let  $ \bar\pi  \circ \varphi_i \stackrel{def }= \pi_i : X_i  \to B$ be  one of the induced elliptic fibrations, 	$\ 0 \leq i \leq 5.$

	\begin{enumerate}

\item $D^i_k$  is a section of the fibration $\pi_i$.
  $D^i_i$ and $D^i_{i+1}$ are independent elements of the free part of the Mordell-Weil group.
		\item  For all $i, j$, $\pi_i^{-1}(p_j)$, the fiber of $\pi_i$  over a singular point  $p_j \in B$ of the discriminant, consists of $3$ rational curves $\mathcal{P}_{j}^{i,A}, \ \mathcal{P}_{j}^{i,B}, \ \mathcal{P}_{j}^{i,0}$.
		\item $\mathcal{P}_{j}^{i,A}, \ \mathcal{P}_{j}^{i,B}, \ \mathcal{P}_{j}^{i,0}$  intersect mutually  transversely at a point (as a fiber of Kodaira type $IV$).  $\mathcal{P}_{j}^{i,0}$ is the strict transform of the cuspidal curve  $\bar{\pi}^{-1}(q_j)$; $\mathcal{P}_{j}^{i,A}$,  $\mathcal{P}_{j}^{i,B}$ are the exceptional $\pu$ for the first and the second blow up respectively.
		\item  If $q$ is a smooth point of the discriminant, $ \pi_i^{-1}(q)$ is a cuspidal curve (Kodaira type $II$).

	\end{enumerate}
\end{thm}
\begin{proof}  	(4) follows from Lemma \ref{Xsing}.
	Theorem \ref{RossiNamikawa},  Lemmas  \ref{Xsing},  \ref{localequ} and  \ref{auto}  provide   the local equations around each singular point as well as the geometric description of  the singular Calabi-Yau and a resolution.  We then can  write  the local equations of the smooth Calabi-Yau,   and of $\mathcal{P}_{j}^{i,A}, \ \mathcal{P}_{j}^{i,B}, \ \mathcal{P}_{j}^{i,0}$. 

	(1) follows from the analysis of these local equations and from Theorem \ref{MW}.
	  A direct computation  in the local equations  proves (2) and (3).
The linear independence of the sections $D^i_i$ and $D^i_{i+1}$ can also be checked explicitly from the intersection numbers in Proposition \ref{intersections}.
	\end{proof}

%Note that 	
The explicit description of the fibration in  Theorem \ref{description}  gives  directly $\chi_{top}(X_i)=36$.

\begin{remark}

In the Namikawa examples studied, both elliptic rational surfaces $B$ and $B'$ in the fiber product $\bar{X} =B \times _{\P^1} B'$ are engineered to have six Type $II$ fibers over the same points, which leads to $6$ isolated  singular points in $\bar X$.
%, whose intersection leads to the  six Type $IV$ fibers of $\bar{X}$.
 The resulting high Mordell-Weil rank of ten   $\mw(\bar{X}/B)$ is a consequence of the fact that the $6$ singular points are non $\mathbb Q$-factorial and that there are no other $\Q$-factorial singularities.
The resolutions produce  two additional independent curve classes in the fiber of the resolved threefold $X_i$, and no (Weil) divisor. Hence the Mordell-Weil group of the resolved threefold $X_i$ is generated by the eight generators present also on a generic Schoen manifold (with $B \neq B'$ general rational elliptic surfaces% ratioanlno Kodaira singular fibers  of Type different than $I_1$ or Type $II$
), together with 
 two more generators associated with two independent rational sections dual to the two additional fibral curve classes from the resolution (in  the Type $IV$ fibers). This is to be compared with the special threefolds studied explicitly in \cite{Schoen87,Morrison:2016lix}  
 with a Mordell-Weil rank of 9: There, $B$ and $B'$ have $I_1$ fibers  over the same $12$ points, which leads to $12$ isolated  non $\mathbb Q$-factorial singular points in $\bar X$.  But the  resolution
 gives rise to one extra curve class in the fiber, leading to $8+1=9$ independent generators of the Mordell-Weil group.
  We believe that the collision of six Type $II$ fibers in the Namikawa threefold
  gives rise to the maximal possible number of independent curve classes in the fiber without inducing a singularity in codimension one, whose resolution would subtract from the Mordell-Weil group.

\end{remark}

\section{The Geometry of the spectrum}\label{sec:GeomSpec}

The  geometry of the Calabi-Yau and its invariants are directly related to the massless particle spectrum. We review the correspondence in  Section  \ref{sec_Anomalies}.  

To define the dictionary between the Spectrum and the geometry, to evaluate the spectrum, 
 the gravitational and gauge anomalies in physics, and the corresponding  formula (\ref{genconj})  in geometry, we need to determine the pairing between $H^2(X_i, \Z) $  and $H_2(X_i, \Z)$, the Shioda map, the height pairings and other  geometric invariants of the Calabi-Yau.

\subsection{Cohomology, homology, pairings, Gopakumar-Vafa invariants
}\label{sec_Cohom}~

From now on we fix a smooth resolution $X_i$ as in Theorem \ref{description} and Theorem \ref{MW} and an index $i$.

\begin{definition}\label{PicB} 

	\begin{enumerate}[label=(\roman*)]
	\item 	Let $f$ and $\mathsf{s}_k, \ 0 \leq k \leq 8$, respectively denote the classes of the fiber, the zero-section and the  generators  of the Mordell-Weil group $\mw({B} / \pu)$; they form  a basis of $H_2(B) $.  
		\item 	Similarly, let $\mathsf{s'}_l, \ 0 \leq l \leq 8$ denote the classes of the  linearly independent sections  of $r': B' \to \pu$  in $H_2(B') $.  
\end{enumerate}	 
\end{definition}

\begin{definition}\label{MWX} \label{9sections}~

 Let $S_l \stackrel{def}= ({\bar \pi'})^*(\mathsf{s'}_l)$, $ \ 0 \leq l \leq 8$, where $\bar \pi' : \bar X = B \times _{\P^1} B' \to B'$.
 We also denote by $S_l$ its isomorphic image  in $X_i$.

\end{definition}

We take  $S_0$  to be  the zero section; the sections  $\{S_1, \cdots, S_8\}$ are independent generators of the Mordell-Weil group $\mw({\bar X} / B)$. 
$S_0$ is then  the  zero section  of the Mordell-Weil group $\mw(X_i/B)$ and  $\{S_1, \cdots, S_8, \Do, \Du \}$  are  independent sections, by  %Definition-Lemma \ref{auto} and 
Lemma  \ref{localequ}.

\begin{definition}\label{NEbasis} Let $\E$ denote the class of the fiber of $\pi_i$, \\
$\skh= S_0 \cdot \p1^*(\mathsf{s}_k),  \ 0 \leq  k \leq 8$,  \\
$ \ \fh= S_0 \cdot \F$,  and \\
$\lkh= S_l \cdot \p1^*(\mathsf{s}_0)$, $\ 1 \leq l \leq 8  $ .
\end{definition}
 We can then conclude:
\begin{proposition}\label{prop-basis}
Fix any index $i$, $0\leq i \leq 5$ and  $j$, $1\leq j \leq 6$. 
With the notation as in Theorem \ref{description} and Definitions \ref{PicB}, \ref{MWX} and \ref{NEbasis}:
		\begin{enumerate}
		\item  $\{{\pi_i}^*(f),  \ {\pi_i}^*(\sfs_k),  \ S_l, \ \Do, \ \Du \}$, with $ \ 0 \leq  k \leq 8, \ 0 \leq l \leq 8$,  is a basis  of the Neron-Severi group  $NS(X_i) \simeq H^2(X_i, \Z) \simeq c_1(Pic(X_i))
		$.
	\item $\{\E, \  \fh ,\   \skh, \  \lkh , \ \PA, \ \PB\}$, $ \ 0 \leq  k \leq 8, \ 1 \leq l \leq 8  $  	is a basis  of  $H_2(X_i, \Z)$.
\end{enumerate}
\end{proposition}
 
\smallskip

%\subsubsection{Pairings}\label{QCI}~

\newpage

\begin{proposition}\label{intersections}  Fix any index $i$, $0\leq i \leq 5$, and  $j$, $1\leq j \leq 6$. % Let $\pi_i:  X_i \to B$.
	With the notation as in Theorem \ref{description} and in Definitions \ref{dik}, \ref{PicB}, \ref{MWX} and \ref{NEbasis}, we find  the following intersection numbers:

	\begin{table}[ht!]
		\begin{center}
			\begin{tabular}{|l|c |c |c |c|c|c|c|c|c|c|} 
				\hline
				& $\E$ &  $ \fh$&  ${\hat{\mathsf{s}}_{0}}$ &  $\skh$ &  ${\hat{\mathsf{s}}_{k'}}$ &  $\lkh$ &  ${\hat{\mathsf{\ell}}_{l'}}$ & $\PA$ &  $ \PB$ &$\mathcal{P}_{j}^{i,0}$ \\\hline  
				$\pi_i^*(f)$  & $0$  &   $0$   &$1$  &$1$ & $1$  & $1$   & $1$ &$0$ &   $0$ &$0$\\\hline  
				$\pi_i^*(\sfs_{0})$ &  $0$  &$1$   &   $-1$     & $0$& ${ 0} $&  $-1$ & $-1 $ & ${0} $& ${0} $& $0$\\\hline  
				$\pi_i^*(\sfs_k)$    &$0$   &   $1$ &   $0$   & $-1 $& ${ 0}$ & ${0} $& ${0} $& ${0} $ & ${0} $& $0$\\\hline 
				$S_{0}$  &  $1$ & $0$ &    $-1 $   & $-1 $    & $-1 $  & ${ 0}$ &${ 0}$&${ 0}$ & ${ 0}$ &  $1$ \\\hline  
				$S_l$ &  $1$ & $0$ & ${ 0}$&${ 0}$ & ${ 0}$ &${ -1}$ &${ 0}$&${ 0}$ & ${ 0}$ & $1$   \\\hline  
				$\Do$ & $1$ &$1$  &$-1$   & ${ 0}$&${ 0}$ & ${ 0}$ & ${ 0}$ &  $-1 $& ${ 0}$ & $2$   \\\hline  
				$\Du$  & $1$ &$1$  &$-1$   & ${ 0}$&${ 0}$ & ${ 0}$ & ${ 0}$ &  $0$& ${ -1}$ & $2$   \\\hline\hline
				$\D3$ &   $1$ &$1$  &$-1$   & ${ 0}$&${ 0}$ & ${ 0}$ & ${ 0}$ &  $1$& ${ 0}$ & $0  $    \\\hline  
				$\Dt$   & $1$ &$1$  &$-1$   & ${ 0}$&${ 0}$ & ${ 0}$ & ${ 0}$ &  $0$& ${ 1}$ & $0$  \\ \hline  
			\end{tabular}
		\end{center}
	\end{table}
In the table, $ k \neq k'$, $1 \leq k, k' \leq 8$ and $ l \neq l'$, $1 \leq l, l' \leq 8$. Above the  double line there are generators of $NS(X_i)$; we will need also the intersections below the double line.

\end{proposition}

\noindent Note that 	$D^i_i \cdot ( \mathcal{P}_{j}^{i,0}+  \mathcal{P}_{j}^{i,A} + \mathcal{P}_{j}^{i,B}     )=D^i_i \cdot \E=1$ as it should be for a section and a fiber (similarly for $D^i_{i+1}$).

\begin{proof}  We need to verify the following intersections:
		\begin{enumerate}
		\item 	$D^i_i \cdot \ \mathcal{P}_{j}^{i,0}=D^i_{i+1} \cdot \ \mathcal{P}_{j}^{i,0}=2 $, 
		\item $D^i_i \cdot \ \mathcal{P}_{j}^{i,A}=D^i_{i+1} \cdot \ \mathcal{P}_{j}^{i,B}=-1$,
		\item  $D^i_i \cdot \ \mathcal{P}_{j}^{i,B}=D^i_{i+1} \cdot \ \mathcal{P}_{j}^{i,A}=0$,
		\item 	$\Dt \cdot \ \mathcal{P}_{j}^{i,0}=D^i_{i+3} \cdot \ \mathcal{P}_{j}^{i,0}=0 $, 
		\item $\Dt \cdot \ \mathcal{P}_{j}^{i,A}=D^i_{i+3} \cdot \ \mathcal{P}_{j}^{i,B}=0$,
		\item  $\Dt \cdot \ \mathcal{P}_{j}^{i,B}=D^i_{i+3} \cdot \ \mathcal{P}_{j}^{i,A}=1$,
		\item $S_k \cdot  \mathcal{P}_{j}^{i,0} =1$, $ \ 0\leq k \leq 8$,
		\item  $S_k \cdot \mathcal{P}_{j}^{i,A}= 0$, $S_k \cdot  \mathcal{P}_{j}^{i,B}=0$, $ \ 0\leq k \leq 8$.
	\end{enumerate} 
(7) and	(8) follow from Lemma \ref{Xsing}. 
Theorem \ref{RossiNamikawa} and  Lemma \ref{localequ} provide the geometric description  and the local equations around each singular point and of $D_i$, $D_{i+1}$, $D_{i+3}$, $D_{i+4} $.  We then can  write  the local equations of the smooth Calabi-Yau, of $\Do, \Du,D^i_{i+3}, \Dt$. For illustration, we exemplify (1), (2), (3) in Appendix \ref{App_intersection}.
(4), (5) and (6) follows from a similar analysis of  these local equations. We note  also that in a neighborhood of  the resolutions of each singular point  $\Do \cap \Du= \mathcal{P}_{j}^{i,A}\cup \mathcal{P}_{j}^{i,B}$

\end{proof}

In  Section \ref{ShiodaHeightP} we verify the  cancellation of the abelian anomalies with the Shioda-map and height pairings.  To that end, we need to describe the intersections of the elements in $NS(X_i)$.

\begin{proposition}\label{lDoDo}  With the same hypothesis as in Proposition \ref{prop-basis}:
	\begin{enumerate}
		\item $S_k \cdot S_k= -\fh$  \, $ \forall k$.
	\item $ S_0\cdot \Do= S_0 \cdot \Du= S_0 \cdot \D3= S_0 \cdot \Dt= \soh $.		
		\item $S_k \cdot \Do= S_k \cdot \Du= S_k \cdot \D3= S_k \cdot \Dt=\mathcal F_k$   is a section of the abelian fibration $X_i \to \pu$ such that ${\pi_{i}}_*(\mathcal F_k)= \sk$.
		\item $D^i_i \cdot \Du  = \soh + \sum_j(\PA+ \PB)$.
			\item $D^i_i \cdot 	\D3= \soh + \hat{\mathcal{C}}$ and $ \quad \Du \cdot 	\Dt= \soh + \hat{\mathcal{C}}$.\\
		$\mathcal{C}={\pi_i}_*(\hat {\mathcal{C}})$ is a smooth curve of genus $4$ such that $[ {\mathcal{C}}]^2=9$, $\mathcal{C} \cdot \sfs_0= 0$ and $\mathcal{C} \cdot f= 3$.
		\item 	$D^i_i \cdot D^i_i  ={2{\pi_{i}}^*(f) \cdot \Do
			%\fh_i 
			+ 3 \soh -\hat {\mathcal{C}}}$.
		\item 	$\Du \cdot \Du  ={2{\pi_{i}}^*(f) \cdot \Du
			%\fh_i 
			+ 3 \soh -\hat {\mathcal{C}}}$.
			\end{enumerate}
In homology: $[{\mathcal{C}}]= [3 \sfs_0 +f)]$. 
\end{proposition}
\begin{proof}
	(1) follows from an argument in  \cite{Friedman:1997yq} (see (7.30) on p. 730). (2), (3) and (4) follow from the analysis of Lemma \ref{auto} and Lemma \ref{localequ}. 
	
	$\Do$ and $\D3$ ($\Du$ and $\Dt$ respectively) intersect  in the zero locus $y=v=0$ in $\bar X$.  The intersection locus has two components, $z=w=0$ and a remaining curve $\bar{\mathcal{C}}$.  The first component is in the resolution $X_i$ in the class $\soh$; the strict transform of $\bar{\mathcal{C}}$,
	$\hat{\mathcal{C}}$, is a smooth curve. Its projection to $B$ intersects the general fiber  $f$ in three distinct points, and in one point at the  six cusps (where $b=0$). That is, $\mathcal{C}$ is a $3:1$ cover of $\pu$ totally ramified at $6$ points. It is then a curve of genus $4$, by the Riemann-Hurwitz formula. The adjunction formula applied to $(\mathcal{C}, B)$, implies that $\mathcal{C}^2=9$. This proves (5).
	
To prove	(6) and (7) we need the following Lemma \ref{linearEquiv} combined with  (1)--(5):
	\end{proof}

\begin{lemma}\label{linearEquiv}  With the same hypothesis as in Proposition \ref{prop-basis}:
	\begin{enumerate}
		\item 
			$D^i_i = 2 \pi_i^*(f) + 2 \pi_i^*(\sfs_0) +2S_0 - \D3$,
				\item 
			$\Du = 2 \pi_i^*(f) + 2 \pi_i^*(\sfs_0)+ 2S_0 - \Dt$.

	\end{enumerate}
	\end{lemma}
 \begin{proof}  We apply  the pairings listed  under the double lines in the Table in Proposition \ref{intersections} and solve the systems.
 	\end{proof}

\begin{proposition}\label{GV}
The { genus-zero} Gopakumar-Vafa invariants on the sublattice of curve classes generated by  $\mathcal{P}_{j}^{i,A}$ and  $\mathcal{P}_{j}^{i,B}$ are
\begin{equation*} n_{\{0,[\mathcal{P}_{j}^{i,A}]\} }=1, \ \  n_{\{0,[\mathcal{P}_{j}^{i,B}]\} }=1,  \ \   \ n_{\{0,[ \mathcal{P}_{j}^{i,A}+\mathcal{P}_{j}^{i,B}]\} }=1  
\end{equation*}
and $0$ otherwise.
\end{proposition}
\begin{proof}
	This follows from \cite{BryanKatzLeung, BryanKarp2005}.  Note that each of these curves is super-rigid \cite[page 291]{BryanPandhar2001}.
	\end{proof}

 \medskip
 
\subsection{The Shioda map and height pairings}\label{ShiodaHeightP}~

\begin{definition} We denote the independent elements of the Mordell-Weil group $\mw(X_i/B)$ 
	as ${\mathbb S}_a$, $a=1, \ldots, 10$ with ${\mathbb S}_l = S_l$, $1\leq l \leq 8$, and ${\mathbb S}_9 = D^i_i$, ${\mathbb S}_{10} = D^i_{i+1}$.
\end{definition}
\begin{definition}\label{defShioda}
With the notation as in Definition \ref{9sections}, the image of the set of independent sections ${\mathbb S}_a$, $1\leq a \leq 10$, within MW$(X_i/B)$ under the Shioda homomorphism
\begin{equation*}
\sigma: {\rm MW}(X_i/B) \to {NS}(X_i) \otimes \mathbb Q
\end{equation*}
introduced in \cite{Shioda2,Wazir,Park:2011ji} is defined to be
\begin{equation*}
	\sigma({\mathbb S}_a) \stackrel{def}= {\mathbb S}_a - S_0 -  {\pi_i}^\ast  \pi_{i \ast} (({\mathbb S}_a- S_0) \cdot S_0) \,.
\end{equation*}
The associated height pairings take the form
\begin{equation*}
b_{a,b}  \stackrel{def}= - (\pi_{i})_\ast (\sigma({\mathbb S}_a) \cdot \sigma({\mathbb S}_b))
\end{equation*}
and are valued in $H_2(B)$.

\end{definition}

Proposition \ref{intersections} enables us to prove  the following Corollaries:
\begin{corollary}\label{intersectionsForAnom}

	With the notation as in Definition \ref{defShioda}, the Shioda map images take the form

           $  \sigma({\mathbb S}_l) = \sigma(S_l) = {S_l-S_0 - {\pi_i}^\ast(f)}$,  \quad $1 \leq l \leq 8$, 
	
	$\sigma({\mathbb S}_9)    = \sigma(D^i_i) =
D^i_i- S_0 -  {{\pi_i}^\ast} (\sfs_0 +f)$,

$\sigma({\mathbb S}_{10}) = \sigma(\Du)= \Du- S_0 -  {{\pi_i}^\ast} (\sfs_0+f)$.

\noindent They have the following %pairwise
 intersections in $X_i$: 

$  \sigma({\mathbb S}_9)   \cdot \sigma({\mathbb S}_9)    {=  \soh - \hat {\mathcal{C}} + \fh +\E}$,

$  \sigma({\mathbb S}_{10})   \cdot \sigma({\mathbb S}_{10})  = 
 \soh - \hat {\mathcal{C}} + \fh +\E$,

$  \sigma({\mathbb S}_{9})   \cdot \sigma({\mathbb S}_{10})  = 
{  \soh 
- \soh - \soh -{\pi_i}^*(f) \cdot \Do\\
 - \soh - \fh  + \soh + \fh\\
 -\soh + \soh  -\E + \E\\
		-  {{\pi_i}^*(f) \cdot \Du}+ \fh + \E=}$\\
	$\phantom{\sigma(D^i_i) \cdot \sigma(\Du)}={   -\soh  + \fh - {\pi_i}^*(f) \cdot \Do - {\pi_i}^*(f) \cdot \Du + \E}$,

$  \sigma({\mathbb S}_k) \cdot   \sigma({\mathbb S}_k) = {-3 {\pi_i}^*(f) \cdot S_k + \fh}$,

$  \sigma({\mathbb S}_k) \cdot   \sigma({\mathbb S}_{k'}) = {-  {\pi_i}^*(f) \cdot S_{k'}- {\pi_i}^*(f) \cdot S_k + \fh}, \quad k \neq k'$,

$  \sigma({\mathbb S}_k) \cdot   \sigma({\mathbb S}_9) = {\skh   -\soh- {\pi_i}^*(f) \cdot S_k -  {{\pi_i}^*(f) \cdot \Do} + \fh  + \E}$,

$  \sigma({\mathbb S}_k) \cdot   \sigma({\mathbb S}_{10})
= {S_k \cdot S_{10}  -\soh  - {\pi_i}^*(f) \cdot S_k-  {{\pi_i}^*(f) \cdot \Du} + \fh + \E}$.

\end{corollary}
\begin{corollary}\label{heightp}
The associated height-pairings are

$b_{9,9} 
=  {- \sfs_0 + {\mathcal{C}} - f}$,

$b_{10,10} 
=  {- \sfs_0 + {\mathcal{C}} - f}$,

$b_{9,10} 
=  { \sfs_0  + f}$,

$b_{k',k} = {  f},  \quad k \neq k'$,

$b_{k,k} = {2f}$,

{{$b_{k,9} = {  \sfs_0 - \sk  +f}$, }}

{{$b_{k,10} = { \sfs_0 -\sk  + f}.$}}

\end{corollary}
\begin{proof}  Note that by construction $(\pi_i)_\ast (\hat{\mathsf{s}}_k)= \sk \in H_2(B)$.
	\end{proof}

\begin{corollary}\label{intersectionsForGaugeAnom}

The only non-vanishing %pairwise 
intersection numbers of the height pairings of Corollary \ref{heightp} 
are, for $1 \leq k,l \leq 8$:

{{$b_{9,9} \cdot b_{k,l} = b_{10,10} \cdot b_{k,l} =  2( 1 +  \delta_{kl})$,}}

{{$b_{9,k} \cdot b_{9,l}      = -(1 + \delta_{kl})$,}}

$b_{9, 9} \cdot b_{9, 9} = b_{10, 10} \cdot b_{10, 10}=  4$,

${b_{9,9} \cdot b_{10,10} = 4}$,

$b_{9,10} \cdot b_{9,10} = 1$,

$b_{9,9} \cdot b_{9, 10} = b_{10,10} \cdot b_{9, 10}   =  2$.

\end{corollary}

\begin{corollary}\label{intersectionsForGrAnom}
The only non-vanishing intersections of the height-pairings of Corollary \ref{heightp} with $(-K_B)$,  the class of the anti-canonical divisor on the base $B$, are

$(-K_B)  \cdot b_{9,9} = (-K_B)  \cdot b_{10,10} = 2$,

$ (-K_B)  \cdot b_{9,10}=  1$.

\end{corollary}
%%%%%%%
\medskip

\smallskip

\section{The spectrum, charges, anomaly cancellation and geometric invariants} \label{sec_Anomalies}

\subsection{General results from F-theory}\label{PhysicsResults}

Compactification of F-theory on $X_i$ gives rise to an effective supergravity theory in six dimensions with $N=(1,0)$ supersymmetry.
Before providing the details of the effective theory, we collect general results for F-theory compactifications on elliptic threefolds that have been derived in the physics literature. 
For derivations and the original references we refer to  the survey articles \cite{Taylor:2011wt,Weigand:2018rez,Cvetic:2018bni}.

For simplicity of presentation and consistently with the Namikawa-Rossi example, we assume that 
%In this subsection we denote by 
$\pi: Y \to B$  is a smooth elliptically fibered Calabi-Yau threefold with base $B$ and zero-section $S_0$. Without loss of generality we assume that the fibration is equidimensional and that $B$ is smooth. We also assume that the Weierstrass model of $Y$,
%which is the resolution of a singular
$\bar Y \to B$
%. For simplicity of presentation we assume that $\bar Y$
 has no singularities appearing in codimension one, that is, in physics language,  the associated non-abelian gauge group associated in F-theory  is trivial.
 % on $B$. 

 We denote by ${\mathbb S}_a$ a set of independent sections in the Mordell-Weil group $MW(Y/B)$ with Shioda map
images  $\sigma({\mathbb S}_a)$ and height-pairings $b_{a,b} = - \pi_\ast(\sigma({\mathbb S}_a) \cdot \sigma({\mathbb S}_b))$, as in Definition \ref{defShioda}.

%\subsubsection{{The spectrum}}~

\smallskip
% We recall the following
\begin{result}[Gauge group and spectrum] {The (abelian) gauge group $G$ of F-theory compactified on $Y$ defined above is $G= \prod_{a=1}^{r} U(1)_a$, where $r$ is the rank of the Mordell-Weil group $MW(Y/B)$.
		The massless physical spectrum comprises
		\begin{enumerate}
			\item  $V=h^{1,1}(Y)-h^{1,1}(B)-1 = {\rm rk}(MW(Y/B))$ vector multiplets,
			\item $T= h^{1,1}(B)-1 $ tensor multiplets,
			\item $H = H_{unch} + H_{ch}$ hypermultiplets, where  $H_{unch} = h^{2,1}(Y) +1$  is the number of uncharged multiplets and $H_{ch}$ the number of hypermultiplets charged under $G$, 
			\item  one universal gravity multiplet.
	\end{enumerate}  }
\end{result}

{ (1), (2) and (4)  immediately provide a correspondence between the massless spectrum and the birational invariants of the elliptic Calabi-Yau $Y$. %  In addition, recall 
	As for (3) we have: 
	\begin{result}[Charged matter multiplicities]\label{Hch}  
		The charged hypermultiplets $H_{ch}$  in (3)
		are in 1-1 correspondence with the %fibral 
		holomorphic curves in the fiber of $Y$ with vanishing intersection with the zero-section $S_0$ (the exceptional fibers of the Weierstrass model).  $H_{ch}$ is computed by either
		\begin{enumerate}[label=\alph*)]  
			\item their Gopakumar-Vafa invariants at genus zero  or 
			\item the localised deformations of the singular fibration $\bar Y \to B$.
		\end{enumerate}
		
	\end{result}    
	\begin{proof} 
		Via duality with M-theory compactified on $Y$, massless hypermultiplets charged under $G$ in F-theory are in 1-1 correspondence with the possible wrappings of M2-branes 
		on the exceptional fibers. %, as reviewed e.g. in \cite{Weigand:2018rez}.
		The Gopakumar-Vafa index of a curve $C$ at genus zero counts the number of hypermultiplets obtained by wrapping M2-branes on $C$ \cite{Gopakumar:1998jq}. See e.g. \cite{Lee:2018urn, Oehlmann:2019ohh, Knapp:2021vkm} for  applications in F-theory on threefolds. The correspondence with 
		the localised deformations of $\bar Y$ follows from \cite{Katz:1996xe}.
	\end{proof}

	\begin{result}[$U(1)_a$ charges] \label{prop-charges-gen}
		The $U(1)_a$ charges of the massless hypermultiplets associated with the exceptional fibers are computed as the intersections of the respective fibers with the Shioda map images $\sigma({\mathbb S}_a)$. 
	\end{result} 
	\begin{proof} {For a derivation via duality with M-theory see \cite{Park:2011ji} as well as the reviews \cite{Weigand:2018rez,Cvetic:2018bni}.     }
	\end{proof}

	\begin{result} [Anomalies]{\cite{Green:1984bx,Park:2011ji}}\label{prop-anomalies-gen}
		The gravitational, mixed gravitational$-U(1)_a-U(1)_b$ and abelian $U(1)_a - U(1)_b - U(1)_c - U(1)_d$ anomalies are cancelled by the six-dimensional Green-Schwarz mechanism if the following equations hold:
		% \cite{Park:2011ji}:
		\begin{eqnarray}
			H - V + 29 T &=& 273  \label{grav-gen}\\
			(-K_B) \cdot b_{a,b} &=&   \frac{1}{6} \sum_I N_I q^I_a q^I_b  \label{grav-U1U1} \\
			b_{a,b} \cdot b_{c,d}  + b_{a,c} \cdot b_{d,b}     + b_{a,d} \cdot b_{c,b}     &=& \sum_I N_I q^I_a q^I_b q^I_c q^I_d  \,.\label{U14}
		\end{eqnarray}
	$ b_{a,b}$ on the   the lefthand side  of (\ref{grav-U1U1}) and (\ref{U14}) is defined in Section \ref{ShiodaHeightP}, Definition \ref{defShioda}.
	\end{result}

	In (\ref{grav-U1U1}) and (\ref{U14}), the righthand side computes the anomaly coefficient for the quartic 1-loop anomalies with
	two and four abelian external legs, respectively, in a six-dimensional $N=(1,0)$ supergravity with $N_I$ massless hypermultiplets
	of $U(1)_a$ charge $q^I_a$.  The lefthand side of (\ref{grav-U1U1}) and (\ref{U14}) represents the contribution to the anomaly from the Green-Schwarz counterterms. 
	
	\subsection{The geometry of the anomaly cancellations}\label{GeomAnom}
From a more general conjecture in 
%\cite{GrassiMorrison03} and 
\cite{GrassiWeigand2} it follows that  for the  smooth elliptically fibered Calabi-Yau threefold $Y \to B$ 
%under general conditions, 
equ. (\ref{grav-gen})
	translates into the relation
	\begin{equation}\label{genconj}
				\boxed{	
			30 K_B^2 + \frac{1}{2} \chi_{top}(Y) =   %&=&% \sum_\alpha(g_\alpha-1) (\dim \adj_\alpha) _{ch} + \sum_\alpha(g_\alpha' -g_\alpha)  (\dim \rho_{0,\alpha})_{ch} \\
		%&& + \sum_Q    (\dim \rho_Q)_{ch} +    
		\sum_{Q'} c_{Q'} \, , }
	\end{equation}
where  $	\sum_{Q'} c_{Q'}=H_{ch} $ are the hypermultiplets  charged only by  the abelian factors $ U(1)_a$.
The  Geometric Anomaly Equation (\ref{genconj}) states that the hypermultiplets  charged only by  the abelian factors
	localise at singular points $Q'$ of the discriminant with multiplicity $c_{Q'}$,  giving $H_{ch}= 	\sum_{Q'} c_{Q'} $.

\section{F-theory on the Namikawa-Rossi threefold}\label{sec:FtheoryNR}

We now apply these general results to F-theory compactified on the Namikawa-Rossi threefold. %by setting $Y = X_i$.

\begin{proposition}\label{spectrum} Let $X_i$ be a smooth minimal resolution of the Namikawa-Rossi threefold and consider F-theory compactified on $X_i$.
The gauge group is a product of ${\rm rk(MW}(X_i/B))= 10$ abelian gauge group factors, $G=\prod_{a=1}^{10} U(1)_a$. Each  $U(1)_a$ gauge potential is associated with the Shioda map image of  one of the independent elements $\{ {\mathbb S}_a\} =\{ {\mathbb S}_l, {\mathbb S}_9, {\mathbb S}_{10} \}$ of MW$(X_i/B)$, as computed in Corollary (\ref{intersectionsForAnom}).
Furthermore
	\begin{enumerate}
\item  $V=h^{1,1}(X_i)-h^{1,1}(B)-1 ={\rm rk(MW}(X_i/B)) =10$,
	\item $T= h^{1,1}(B)-1 = 9$,
	% \item \todo{The universal gravitational multiplet}
	\item $H_{unch} = 4$ , $H_{ch}=18$ and $H = H_{unch} + H_{ch}=22$.
\end{enumerate}
\end{proposition}
}
\begin{proof} (1)  and (2) follow by constructions and from the Shioda-Wazir formula;  $h^{2,1}(X_i) +1 =4$ by  Theorem \ref{RossiNamikawa}.
	The holomorphic curves in the fiber of $X_i$ with vanishing intersection with the zero-section $S_0$ (the exceptional fibers) are components of  the  fibers of the points of  the singular locus of the discriminant: $\pi _i ^{-1} (p_j)= ( \mathcal{P}_{j}^{i,0}+  \mathcal{P}_{j}^{i,A} + \mathcal{P}_{j}^{i,B}     )$, $1 \leq j \leq 6$, with the notation as  in Theorem \ref{description}. 
	Each such fiber contains 3 such holomorphic curves in class $\mathcal{P}_{j}^{i,A}$,  $\mathcal{P}_{j}^{i,B}$ and $\mathcal{P}_{j}^{i,A} + \mathcal{P}_{j}^{i,B}$.   By Result \ref{Hch} (a),  their genus-zero Gopakumar-Vafa invariants  of Proposition \ref{GV} invariants compute $H_{ch}$. 
	{Each of the six singularities of the singular fibration $\bar X$ defined in Theorem \ref{RossiNamikawa} can be deformed to  $3$ nodes \cite[Proposition 7]{RossiMIchele20174NamikawaEx}.  }
Each node contributes $+1$ to $H_{ch}$, yielding $H_{ch} = 3 \times 6 = 18$ as well, by  Result \ref{Hch} (b). The deformation do not lift to global deformations of the resolution $X_i$ \cite{NamikawaStratified,RossiMIchele20174NamikawaEx}.
	\end{proof}

\begin{proposition} \label{Charges-NR}
 Let $X_i$ be the Namikawa-Rossi threefold.  Let 
 $q_a$ denote the $U(1)_a$ charges for the hypermultiplets associated with the exceptional fibers $\mathcal{P}_{j}^{i,A}$,  $\mathcal{P}_{j}^{i,B}$ and $\mathcal{P}_{j}^{i,A} + \mathcal{P}_{j}^{i,B}$. Then the non-zero $U(1)_a$ charges are computed as the respective intersections with the Shioda map images $\sigma({\mathbb S}_a)$:
 \begin{center}
 	\begin{tabular}{ c |c c c  } 
 		&  $\mathcal{P}_j^A$ & $\mathcal{P}_j^B$   & $\mathcal{P}_j^A +\mathcal{P}_j^B$  \\\hline  
 		$q_9$         & -1 & \ { 0} & { -1}  \\ 
 		$q_{10}$ & \ { 0} & -1 & { -1}  \\ 
 		$q_{l} $      &  \ { 0} & \ { 0}   & \ {0} \\ 
 	\end{tabular}
 \end{center}
\end{proposition}
\begin{proof} We apply Proposition \ref{intersectionsForAnom} and \ref{intersections} to evaluate the charges as in Result \ref{prop-charges-gen}.
	\end{proof}

\begin{proposition}\label{NR-anS}
F-theory on the Namikawa-Rossi manifold satisfies the anomaly cancellation conditions as collected in Result \ref{prop-anomalies-gen}.
\end{proposition}
\begin{proof} We evaluate the purely gravitational and the abelian and mixed gravitational-abelian anomaly conditions in turn.

\noindent {\it Gravitational Anomalies}~
The condition for cancellation of the purely gravitational anomalies, equ. (\ref{grav-gen}), is manifestly satisfied because $H=H_{unch} + H_{ch} = 4+18=22$, $V=10$ and $T=9$.

\smallskip

\noindent {\it (Mixed) Abelian anomalies}\label{aa} ~
On the righthand side of  (\ref{grav-U1U1}) and (\ref{U14}), applied to F-theory on $X_i$, the index $I$ becomes a multi-index $I=(C,j)$, where $C \in \{A, B, A+B\}$ and $j \in \{1, \ldots, 6\}$
label the curves $\mathcal{P}_{j}^{i,A}$,  $\mathcal{P}_{j}^{i,B}$ and $\mathcal{P}_{j}^{i,A} + \mathcal{P}_{j}^{i,B}$
appearing in the table in  Proposition \ref{Charges-NR}.
$N_I$ counts the number of massless hypermultiplets associated with each of these curves and coincides, by  Result \ref{Hch} (a),
with the corresponding genus-zero Gopakumar-Vafa invariant computed in Proposition \ref{GV}.

With this and the charges as in the table in Proposition \ref{Charges-NR}, and  $1 \leq l \leq 8$, $1 \leq a \leq 10$, equ. (\ref{grav-U1U1}) becomes the requirement that  
\begin{eqnarray}
U(1)_9^2 -{\rm grav}: & (-K)_B  \cdot b_{9,9}=  2  \\
U(1)_9-U(1)_{10} -{\rm grav}: & (-K)_B  \cdot b_{9,10} = 1 \\
U(1)_{10}-U(1)_{10}-{\rm grav}: & (-K)_B  \cdot b_{10, 10}= 2  \\ 
U(1)_{l}-U(1)_{a}-{\rm grav}: & (-K)_B  \cdot b_{l, a}= 0  \,,
\end{eqnarray}
and equ. (\ref{U14}) becomes
\begin{eqnarray}
U(1)_9^4: & b_{9,9} \cdot b_{9,9} =  4   \\
U(1)_9^3 - U(1)_{10}: & b_{9,9} \cdot b_{9,10} =  2    \\
U(1)_9^2  -U(1)_{10}^2: & b_{9,9} \cdot b_{10,10}  + 2 b_{9,10} \cdot b_{9,10}   =  6    \\
U(1)_9 - U(1)_{10}^3: & b_{9,10} \cdot b_{10,10}   =  2    \\     
U(1)_{10}^4: & b_{10, 10} \cdot b_{10, 10}  =  4     \\
U(1)_{l} - U(1)_a -U(1)_b - U(1)_c: & b_{l, a} \cdot b_{b, c} +  b_{l, b} \cdot b_{c, a} +  b_{l, c} \cdot b_{a, b}  =  0  .
  \end{eqnarray}

These equations are manifestly satisfied with the help of Corollaries \ref{intersectionsForGrAnom} and \ref{intersectionsForGaugeAnom}.
\end{proof}

\begin{proposition}\label{GeomAnomNamRossi}
The Namikawa-Rossi manifolds satisfy  the  Geometric Anomaly Cancellation equation (\ref{genconj}).
\end{proposition}
\begin{proof} Indeed,  $K_B^2 =0$, $\chi_{top}(X_i)=36$ by Theorems \ref{RossiNamikawa},  \ref{MW} and 
	$ \sum_{Q'} c_{Q'}= 6 \times 3=18$, by Proposition \ref{GV}.
	\end{proof}

\begin{corollary}\label{birinv}
$ \sum_{Q'} c_{Q'}$ is a birational invariant of the minimal model of the elliptic fibration. 
\end{corollary}
\begin{proof} In fact the left hand side of the equation (\ref{genconj}) is a birational invariant of the minimal model \cite{GrassiWeigand2}.
	\end{proof}

$ \sum_{Q'} c_{Q'}$ is a birational invariant of the non $\Q$-factorial  terminal singularities of the Weierstrass model $\bar X$, in the sense that it is a birational invariant of the $\Q$-factorialization.
\medskip

\section{{The Weierstrass model over $\P^2$, Elkies' birational example}} \label{sec_Elkies}~
%As pointed out in the Introduction, in 2018

In this Section we take the first steps in addressing the question of  whether the model  $W_{\NDE} \to \P^2$  constructed in \cite{Elkies} is  birationally Calabi-Yau.
We prove that  the Weierstrass models over $\pd$  of the Namikawa-Rossi  threefolds  are not the ones constructed by Elkies.

\subsection{Summary of \cite{Elkies}:} In the 2018 seminar talk \cite{Elkies}  Elkies  gave a construction of 
a family of  elliptically fibered threefolds in Weierstrass  form,  $W_{\NDE} \to \P^2$, with $\rk(\mw(W_{\NDE} /\P^2))=10$ and  $K_ {W_{\NDE}} \equiv 0$.  
\cite{Elkies} does not address the question of whether  the minimal  resolutions are Calabi-Yau threefolds.

The starting point of the construction is what Elkies calls an  ``excellent family", that is  elliptic fibrations which  depend on the parameter $\zeta$:

\begin{equation}\label{ElkiesEquation}
	y^2= x^3+ ({p}_4\zeta^4 + {p}_{10}\zeta)x + \zeta^9 + {p}_6 \zeta^6+ {p}_{12}\zeta^{3}+ {p}_{18} \,.
\end{equation} 
In Elkies' construction the variables $(x,y,\zeta)$ have weights $(6,9,2)$ %respectively
and the coefficients ${p}_j$ are the  invariant forms  of degree  $j$ in $\P^4$ for  the Shephard-Todd unitary reflection group  $\ST_{33}$ in $\C^5$ \cite{ShephardTodd}, which we discuss below.
 Then Elkies obtains elliptic  threefolds  ${W_{\NDE}}$ by restricting the coefficients ${p}_j$  to a general $\P^2$ and  by  taking $\zeta$ to be a quadratic form in that $\P^2$.

The discriminant locus  of each fibration $W_{\NDE} \to \P^2$  is then a curve of  degree $36$,  and $K_ {W_{\NDE}} \equiv 0$;  $h^1(\mathcal O_{W_{\NDE} })=0,h^2(\mathcal O_{W_{\NDE} })=0$ by construction. The threefolds $W_{\NDE}$ are potentially birational Calabi-Yau. However, it is easy to construct  Calabi-Yau Weierstrass models with the same numerical properties
with log canonical singularities which are not birationally equivalent to a Calabi-Yau with terminal singularities. The example of \cite{Elkies} might a priori fall into this class.

Elkies' excellent  family extends Shioda's excellent families for rational elliptic surfaces. Here ``excellent" refers to the explicit generators  of  the Mordell-Weil group of sections \cite{ShiodaSchuettBook}. The particular structure of the excellent family implies that $\rk \MW(W_{\NDE}/ \P^2)=10$.

The coefficients $p_j$  are of geometric  interest in their own right; in fact $\ST_{33} \simeq  \Z/{2\Z} \times \Sp(4, \mathbb F_3)$,   where $\Sp(4, \mathbb F_3) \simeq G_{25920}$ is the Burkhardt group
\cite{Hunt96}.
Shephard and Todd \cite{ShephardTodd} prove that the invariants $p_j$   of  $\ST_{33}$  are the same invariants as for  the Burkhardt group $G_{25920}$. The latter were originally computed by Burkhardt \cite{Burkhardt1881}.  In particular  Burkhardt  shows that possible coefficients ${p}_{18}$   are either  the product of lower degree invariants %of degree $4, 6$ and $12$ or a %\todo{(check the Deutsch) } 
or an irreducible polynomial of degree $18$, or a linear combination thereof.

As will become clear in the following section, 
Elkies' special choice $\zeta =0$  in the family (\ref{ElkiesEquation}) could 
be a candidate for the Weierstrass model of the Namikawa-Rossi manifolds. We will now give an explicit construction of the Weierstrass models
and then compare to Elkies' model for $\zeta =0$.

\subsection{Weierstrass models over $\P^2$ of the Namikawa threefolds}\label{WP}~

\begin{proposition}\label{Model} Let $ X_i$  be one  fixed  smooth resolution of the Namikawa threefolds as in Theorem \ref{RossiNamikawa}  with elliptic fibration $\pi_i : X_i \to B$. Let ${\bar \pi}_i$
be the morphism induced by the contraction $B \to \pd$ of the rational curves $\{ \sfs_0, \cdots , \sfs_8\}$:
\[
\begin{tikzcd}
	X_i \arrow[d, "\pi_i"']  \arrow[dr, "{\bar \pi}_i"] & \\
	{B} \arrow[r] &  \pd   \\
\end{tikzcd}
\]	

Then there exists a diagram
\[
\begin{tikzcd}
	X_i \arrow[d, "\pi_i"']  \arrow[dr, "\bar \pi_i"'] \arrow[r, dashed, "\psi_1"]& X'_i  \arrow[l, dashed]  \arrow[d, ]  \arrow[r,  "\psi_2"]&     Z_i \arrow[dl, "\pi_{Z_i}"]  \\
	{B} \arrow[u, bend left=50] \arrow[r]  & \pd \arrow[u, bend left=50, shift right=0.8ex]   \arrow[ur, bend right=40 , shift right=0.6ex] &  \\
\end{tikzcd}
\]
such that the following holds:

\begin{enumerate}
	\item $ \psi_1:  X_i \dasharrow X'_i$, with $X_i'$ smooth, is a birational map constructed as  the composition of  the  81 flops of the rational curves $ {\pi_i}^*(\sfs_k) \cdot S_\ell, \ 0 \leq \ell \leq 8, \  0 \leq  k \leq 8$. 
	The
	% support of the 
	discriminant locus of $\bar \pi_{i}$ consists  of  6 irreducible cuspidal curves which intersect pairwise transversely in $9$ distinct smooth points $\{z_0, \cdots z_8\}$.  The fiber over each $z_j$ is  the surface
	${\psi_1}_*(\p1 ^*(\sfs_k)) \simeq \pd$,  $0 \leq k \leq 8$.
	\item  The elliptic fibration $X'_i \to \pd$ has $11$ linearly independent sections (i.e. the rank of ${\rm  MW}(X_i'/\mathbb P^2)$ is $10$), the strict transforms of the sections of $\pi_i$:  \\
	$S'_{ l, \pd}\stackrel{def}= {\psi_1}_*(S_\ell), \ 0 \ \leq \ell \leq  8$, $\ {D'}^i_i  \stackrel{def}={\psi_1}_*(\Do)$ and  ${D'}^i_{i_1} \stackrel{def}={\psi_1}_*(\Du)$ . 
	\item  $\psi_2: X'_i \to Z_i$ is a composition of $9$  birational contractions  with exceptional loci 
	$\{ {\psi_1}_*(\p1 ^*(\sfs_k)) \simeq \pd, 0  \leq k \leq 8 \}$.
	The Calabi-Yau $Z_i$ has $9$ canonical (but not terminal) isolated singularities.  
	\item $Z_i$ is rigid.
	\item The elliptic fibration $Z_i \to \pd$ has $11$ linearly independent sections  (i.e. the rank of ${\rm  MW}(Z_i/\mathbb P^2)$ is $10$).
	\item The elliptic fibration $\pi_{Z_i}:  Z_i \to \pd $ is equidimensional.  
	\item  For every $ k$, the singular fiber $\pi_{Z_i}^{-1}(z_k)$  consists of  $9$ rational curves meeting at  the  point of canonical singularity of $Z_i$.
\end{enumerate}

\end{proposition}

\begin{proof}The  statements (1) and (3)  follow from the contraction theorems and the existence of log flips for threefolds,  stated in Appendix \ref{app_background} for convenience: To obtain  the flops in  (1) in Theorem \ref{flop}  we take $Y=\pd$,   $\mathcal D= \epsilon  {\pi_i}^*(\sfs_k)$ for a fixed $k, \  0 \leq k \leq 8$, $\epsilon \ll 1$ and $R=  {\pi_i}^*(s_k) \cdot S_\ell, \ 0 \leq \ell \leq 8$. Each  of these log-flips is a flop.
For each of the contraction morphism in (3), we take in Theorem \ref{Contr} $\mathcal D = {\psi_1}_*(\pi_i^*(\sfs_k)) \simeq \pd$ and $R$ any line in $\pd$. (4) follows from \cite{SchlessingerInv1971}, \cite{AltmannVersal} and the survey \cite{RossiGeomTrans}. 
(2), (5), (6) and  (7) follow from the construction.
\end{proof}

We now give an intrinsic description of the Weierstrass model over $\mathbb P^2$ of the Namikawa-Rossi manifolds.

\begin{lemma}\label{NakayamaW} Let  $S'_{ 0, \pd}= {\psi_1}_*(S_0)$ be a fixed section  for $\bar \pi'_i: X'_i \to \pd$.
There exists a crepant  birational morphism $\psi_3$ such that the following diagram commutes:
\[
\begin{tikzcd}
	X'_i   \arrow[d, "\bar \pi_i' "]  \arrow[r,  "\psi_3"] &    W_{\pd}\arrow[dl, "\pi_{W_{\pd}}"]  \\
	\pd &   \\
\end{tikzcd}
\] 
$ \pi_{W_{\pd}}: W_{\pd} \to \pd$ is the Weierstrass model  of $X'_i$ with marked section $S_{\pd}\stackrel{def}={\psi_3}_*S'_{ 0, \pd}$. 

In addition, $K_{ X'_i }+S'_{ 0, \pd}={\psi_3}^*(K_{ W_{\pd}}+ S_{\pd} )$.
\end{lemma}

\begin{proof} The existence of the Weierstrass model and of the commutative diagram  such that $S'_{ 0, \pd}={\psi_3}^*( S_{\pd} )$ is proved in \cite{Nakayama88}.  The morphism $\psi_3$ is crepant because, with $\Lambda_{\mathbb P^2}$ the support of the discriminant, $\mathcal O_{X'_i} \simeq K_{X'_i} \simeq ( {\bar \pi'_i} )^* (K_{\pd}+ \Lambda_{\pd})$ and  $\mathcal O_{W_{\pd}}\simeq  ( {{\bar \pi }_{W_{\pd}} }) ^*(K_{\pd}+ \Lambda_{\pd}) \simeq K_{W_{\pd}}$  since  $\bar \pi' _i$ and $\pi_{W_{\pd}} $ have the same discriminant. 
\end{proof}

The construction of the Weierstrass model in Lemma \ref{NakayamaW} is not explicit, so we use the construction of the  relative log canonical model instead.

\begin{thm}\label{lc} 	Let  $h$ be a general line in $\pd$,  $F'_{\pd}\stackrel{def}= (\bar \pi'_i)^* (h)$ and $0 < a  \leq 1$.
The  Weierstrass model  $W_{\pd} \to \pd$ %with marked section $S_{\pd}$
of the Namikawa-Rossi threefold is the relative log canonical model  of $({{X'_i}}, S'_{0, \pd}+aF'_{\pd} ) $ described in Proposition \ref{Model}. 
$W_{\pd}$ is obtained from $X_i$  by the composition of $\psi_1$, $\psi_2$ and 
the birational contractions of  the  flops of the rational curves $ {\pi_i}^*(s_k) \cdot S_\ell, \ 0 \leq \ell \leq 8, \  0 \leq  k \leq 8$,  the $6$ pairs of curves $\{\PA, \PB\}$. % and the 9 surfaces 		${\psi_1}_*(\p1 ^*(s_k)) \simeq \pd$,  $0 \leq k \leq 8$.

$W_{\pd}$  is also the  relative log canonical model  of $({{Z'_i}},  {\psi_2}_*S'_{0, \pd}+a{\psi_2}_*F'_{\pd})$.

\[
\begin{tikzcd}
	& X_i \arrow[d, dashed, "\psi_1"] &\\
	& 	({{X'_i}}, S'_{0, \pd}+aF'_{\pd} ) \arrow[u, dashed, shift left]   \arrow[dl, "{\psi_3}"'] \arrow[d, "{\psi_2}"] \arrow[dr]  &\\
	(W_{\pd}, {\psi_3}_*S'_{0, \pd}+a{\psi_3}_*F'_{\pd}) \arrow[dr, "\pi_{W_{\pd}}"']  &	({{Z'_i}},  {\psi_2}_*S'_{0, \pd}+a{\psi_2}_*F'_{\pd})\arrow[d, "{\pi_{Z_i}}"]  \arrow[l,dashed, "{\phi}"']\arrow[r,"{\psi_4}"] &( Z^{\ell c},   {(\psi_4 \psi_2)_*S'_{0,\mathbb P^2} + a {(\psi_4 \psi_2)}_*F'_{\pd}}) \arrow[dl, "\bar{\pi}"]\\
	& \pd &\\
\end{tikzcd}
\]

\end{thm}
\begin{proof}  For $ 0 \leq a \leq 1$,  $(X'_i, S'_{0, \pd}+a F'_{\pd})$ is  a log canonical pair. General results from the minimal model program together with the existence of abundance in dimension $3$ \cite{KeelMatsukiMcKernanAbundance} ensure the existence of the log canonical model $(Z^{\ell c},S^{\ell c}+aF^{\ell c}))$  for the pair $(X'_i , S'_{0, \pd}+aF'_{\pd})$, $0\ < a \leq 1$,  relative to the fibration $\bar \pi_i'$. %The log canonical model  is unique \cite[Theorem 6.32]{KollarMori}.
$K_{Z'_i}+ {\psi_2}_*S'_{0, \pd}+a{\psi_2}_*F'_{\pd} \ $ is ${\pi_{Z_i}}$-nef.  Abundance \cite{KeelMatsukiMcKernanAbundance} gives  the  birational morphism  $\psi_4$ to the log canonical model $
( Z^{\ell c},   {\psi_4}_*S'_{0, \pd}+a{\psi_4}_*F'_{\pd})$	(Definition \ref{lcmodel}). $\psi_4$ contracts the  flops of the rational curves $ {\pi_i}^*(s_k) \cdot S_\ell, \ 0 \leq \ell \leq 8, \  0 \leq  k \leq 8$ and  the $6$ pairs of curves $\{\PA, \PB\}$.
$K_{W_{\pd}}+{\psi_3}_*S'_{0, \pd}+a{\psi_3}_*F'_{\pd} $ is $\pi_{W_{\pd}}$-ample.		$({{X'_i}}, S'_{0, \pd}+aF'_{\pd} )$ is a common  log resolution of the three log canonical  pairs  $(W_{\pd}, \ {\psi_3}_*S'_{0, \pd}+a{\psi_3}_*F'_{\pd}) $, $( Z'_i,  \  {\psi_2}_*S'_{0, \pd}+a{\psi_2}_*F'_{\pd})$ and $ ( Z^{\ell c},  \  {\psi_4}_*S'_{0, \pd}+a{\psi_4}_*F'_{\pd}) $.
The morphisms $\psi_2$,  $\psi_3$ and $\psi_4$ are isomorphisms onto their images when restricted to $ S'_{0, \pd}+aF'_{\pd}$. 		Then 
$ K_{X'_i} + S'_{0, \pd}+aF'_{\pd} \simeq {( \psi_2)}^*(K_{Z'_i}+ {\psi_2}_*S'_{0, \pd}+a{\psi_2}_*F'_{\pd} )$ and 
$ K_{X'_i} + S'_{0, \pd}+aF'_{\pd} \simeq {(\psi_4 \cdot \psi_2)}^*(K_{Z^{\ell c}}+{(\psi_4 \psi_2)}_*S'_{0, \pd}+a{(\psi_4 \psi_2)}_*F'_{\pd} )$ by construction while 
$K_{X'_i} + S'_{0, \pd}+aF'_{\pd} \simeq {\psi_3}^*(K_{W_{\pd}}+{\psi_3}_*S'_{0, \pd}+a{\psi_3}_*F'_{\pd} ) $ by Theorem \ref{NakayamaW}. In particular  $(W_{\pd}, {\psi_3}_*S'_{0, \pd}+a{\psi_3}_*F'_{\pd}) $ satisfies the conditions to be a log canonical model, Definition \ref{lcmodel}. We conclude  as in Section I.4.1. in \cite{Zanardini2021} by recalling  that the log canonical model is unique  \cite[Theorem 3.52 ]{KollarMori}.
\end{proof}

Summarizing:

\begin{corollary}\label{Elkies} 
$\pi_{W_{\pd}}:  {W_{\pd}} \to \pd $ has affine equation $y^2=x^3+\beta(s,t)$, where $\beta(s,t)$
%, a section of $\mathcal O_{\pd}(18)$, 
is the equation of the $6$ general cuspidal curves in the pencil of $\pd$ which give rise to the smooth general rational elliptic surface with  $6$ type $II$ fibers  $r: B \to \pu$. The $6$ cuspidal curves intersect in  the points $\{ z_0, \cdots , z_8 \}$.
The Weierstrass model is non-minimal of type $(*,6,12)$ at each of the points   $  \{ z_0, \cdots , z_8 \} \subset \P^2$.  
${W_{\pd}}$ has  $\mathbb Q$-factorial canonical, but not terminal singularities in the fibers over  $  \{ z_0, \cdots , z_8 \} \subset \P^2$. 
${W_{\pd}}$ has non $\mathbb Q$-factorial terminal singularities in the fibers over the $6$ cuspidal points. 
The singular locus of the reduced discriminant consists of $45$ points.
\end{corollary}
\begin{proof}
The affine equation is $y^2=x^3+\beta(s,t)$ because  $j( {W_{\pd}} )=0$. The zero locus of $\beta(s,t)$ is the reduced discriminant, which by (1) in Proposition \ref{Model} and Theorem \ref{lc}  consists of  the $6$ type $II$ fibers in pencil in $\pd$ which give rise to the smooth general rational elliptic surface  $r: B \to \pu$.
The type of the Weierstrass model  then follows, in fact  if  $y^2=x^3+ \alpha x + \beta$  is a local Weierstrass equation and $\delta$ is the equation for the discriminant  then  the triplet $(\nu (\alpha(P)), \nu (\beta(P)), \nu (\delta(P)))$ is given by the vanishing orders at $P$ of $\alpha, \beta$ and $\delta$.  It is non-minimal by definition%  since $(*,6,12) \geq (*, 6, 12)$
.
The contraction $\psi_2$ gives rise to canonical but non-terminal singularities, by part (3) in Proposition \ref{Model}, while the contraction $\psi_4 $ results in non $\mathbb Q$-factorial terminal singularities (see the proof of Theorem \ref{lc}).
\end{proof}

\smallskip

\subsection{Comparison with Elkies' construction}\label{ElkiesConstruction}~

We now compare the Weierstrass model ${W_{\pd}}$ of the Namikawa-Rossi threefolds, which we described  explicitly in Theorem \ref{lc} and Corollary \ref{Elkies},
to Elkies' Weierstrass model (\ref{ElkiesEquation}) for $\zeta=0$, $W_{\NDE,0}: y^2= x^3+{p'}_{18}$. 
It is clear that if $ {p'}_{18} $ is taken to be irreducible, the two Weierstrass models are different.
For more general invariants ${p'}_{18}$ one must answer the question whether the defining equation $\beta(s,t)$ appearing in ${W_{\pd}}$ in Corollary \ref{Elkies}
is the restriction of an invariant of the Burkhardt group to $\mathbb P^2$.
We pursue this investigation in an upcoming paper \cite{GVZ2021}.

\appendix

\section{Review of background material} \label{app_background}
We review some foundational results in birational geometry which can be found for example in  
\cite{KollarMori}.  Applications to relative  log canonical models of elliptic fibrations can be found in  Chapter I  of \cite{Zanardini2021}.
\begin{thm}[Contraction morphism]\label{Contr} Let $\pi: Z \to Y$ be a morphism, $Z$ a threefold, $\mathcal D$ an effective $\mathbb{Q}$-divisor. If $(Z, \mathcal{D})$ has $\Q$-factorial 
	klt singularities and $K_{Z}+ \mathcal{D}$ is not $\pi$-nef, that is $(K_{Z} + \mathcal{D}) \cdot  R < 0$, for some extremal ray $R \in NE({Z} /B)$, then there exists a morphism $\bar \phi : {Z} \to \bar Z$, contracting all the curves in the numerical equivalence (homology) class of $[R]$ such that the following diagram is commutative:
\[
\begin{tikzcd}
({Z}, \mathcal{D} ) \arrow[d, "\pi"'] \arrow[r,"\bar{\phi}"]& (\bar Z,  \mathcal{\bar D}) \arrow[dl, "\bar{\pi}"]\\
Y &\\
\end{tikzcd}
\]
$\bar Z$ is a normal variety and $\dim \NE(Z/B) > \dim \NE(\bar Z/B)$.

\end{thm}

\begin{thm}[The flops]\label{flop} Let  $(Z, \mathcal{D})$  a variety with $\Q$-factorial 
	klt singularities. Let   $\bar{\phi}$  be a  $(K_{Z} +\mathcal{D})$ contraction of an extremal ray $R$ as in Theorem \ref{Contr}. 
Assume that $\bar{\phi}$ is small. Then there exists a log flip $ \psi :  ({Z},  \mathcal{D} )  \dasharrow ( Z',   \mathcal{\bar D'})$  of $R$.  That is, 
$K_{Z'}+  \mathcal{D'}$ is $\pi$-nef (i.e. 
 $(K_{Z'} + \mathcal{D'}) \cdot R' > 0$, $\forall \ R' \in \NE(Z'/ \bar Z)$)
  and the following diagram is commutative
\[
\begin{tikzcd}
({Z},  \mathcal{D} ) \arrow[d, "\bar \phi"] \arrow[r, dashed, "\psi"]& ( Z',   \mathcal{\bar D'})  \arrow[l, dashed] \arrow[dl, "\bar {\phi'}"]\\
(\bar Z,   \mathcal{\bar D}) &\\
\end{tikzcd}
\]
 $(Z',  \mathcal{D'})$ has $\Q$-factorial 
klt singularities.
\end{thm}

There is also a relative version.
\begin{definition}
	Let  $Z, Y$  be normal varieties,   $f: X \to Z$ a birational morphism and $(Z, \mathcal D)$ a pair such that $K_Z+ \mathcal D$ is $\Q$-Cartier. Let $\{ E_j\} $ be the collection  of the  exceptional divisors; then the formula $$K_Y+(f^{-1})_*(\mathcal D) \equiv f^* (K_X+ \mathcal D) + \sum_{j}a(E_j,Z, \mathcal D)E_j$$ defines $a(E_j,Z, \mathcal D)$.
	
	$({Z}, \mathcal{D} ) $ is  a log canonical  pair if and only if  $\inf_{j}a(E_j,Z, \mathcal D) \geq -1$.
	
\end{definition}

\begin{definition}\label{lcmodel} Let $({Z}, \mathcal{D} ) $ be a log canonical pair and $\pi: Z \to Y$ a proper morphism.  $( Z^{\ell c},  \mathcal{ D}^{\ell c})$ is the log canonical model over $Y$ if in the following diagram:
	\[
	\begin{tikzcd}
	({Z}, \mathcal{D} ) \arrow[d, "\pi"'] \arrow[r, dashed, "{\phi}"]& ( Z^{\ell c},  \mathcal{ D}^{\ell c}) \arrow[dl, "\bar{\pi}"]\\
	Y &\\
	\end{tikzcd}
	\]
	\begin{enumerate}
		\item $\bar \pi$ is proper
		\item $\phi ^{-1}$ has no exceptional divisor
\item $\phi_*(\mathcal D)= \mathcal D^{\ell c}$
\item $K_{Z^{\ell c}}+ \mathcal{D}^{\ell c}$ is $\bar \pi$-ample
\item for every $\phi$-exceptional  divisor $E \subset Z$,  $a(E,Z, \mathcal D) \leq a(E,Z^{\ell c}, \mathcal D^{\ell c})$.
	\end{enumerate}
	\end{definition}

\section{Derivation of intersection numbers} \label{App_intersection}

In this appendix we exemplify the derivation of the intersection numbers presented in Proposition \ref{intersections}. These intersections are in a neighborhood of the resolution of each singular point. We derive  (1), (2) and (3) in the proof   using geometry,  the local equations around a singular point of the threefold and its resolution around the exceptional loci. 

With the notation from Theorem \ref{description} and Definition \ref{dik} %we consider, as an example,
 we recall that the section $D^i_{i+1}$ on $X_i$ is by construction  the strict transform  of  $D_{i +1}$ on $\bar X$  %with  respect to 
 by the resolution $\varphi_i$.  In a neighborhood   of  the exceptional loci, $D^i_i \cap D^i_{i+1} = \cup_j \PA \cup \PB$. We note also that 
  $D_{i +1}$ inherits from $r: B \to \P$  the structure of a  rational elliptic surface with % general elliptic fiber, which we denote by $E_{\Du}  $, and 
 six fibers of type $II$. ${\varphi_i}|_{\Du}$ induces
two  blow ups of  the type $II$ fibers at the cuspidal points.  $\Du$ is  then a  non minimal rational elliptic surface %with general fiber    $E_{\Du} $ 
with exceptional curves $\cup_j \PA \cup \PB$.  Let 
% singularity in the elliptic fiber of $D_{i+1}$. 
$E_{0, \Du}  $ be the strict transform of the cuspidal fiber in $\Du$, and  $E_{\Du} $  be the general fiber (note that $\pi_{i} (E_{\Du} )= f \in B$).

Then  
 $ E_{\Du}  \equiv (E_{0, \Du} + 2 \PA + 3 \PB)|_{D^i_{i+1}}$ with 
 $(\PA \cdot \PA )_{|\Du}= -2$,  $(\PB \cdot \PB)_{|\Du}= -1$,  for any $1 \leq j \leq 6$.
The three component curves $E_{0, \Du}$, $\PA$ and $\PB$ intersect in one point.

Hence we obtain the following intersection numbers:
\bea
D^i_i \cdot \PA &=& ({D^i_i \cdot \PA})_{|\Du} = \frac{1}{2} (\PA + \PB) \cdot ( E_{\Du}  - E_{0, \Du} - 3 \PB)|_{D^i_{i+1}} \nonumber \\
& =& \frac{1}{2} (0 -2-3 +3) = -1  \,, \nonumber \\
D^i_i \cdot \PB &=& \frac{1}{3} D^i_i \cdot (E_{\Du}  - E_{0, \Du}  - 2 \PA)|_{D^i_{i+1}}  \nonumber\\
&=& \frac{1}{3} (\PA + \PB) \cdot (E_{\Du}  - E_{0, \Du} - 2 \PA)|_{D^i_{i+1}} \nonumber \\
& =& \frac{1}{3} (0 -2-2(-2+1)) = 0  \,.\nonumber 
\eea
Either from the local equations of the  resolved Calabi-Yau,   or from the above intersections together with $D^i_i \cdot {\mathcal E}=1$  we find also
\bea
D^i_i \cdot \mathcal{P}_{j}^{i,0} = 2 \,.  \nonumber 
\eea

The intersection numbers with $D^i_{i+1}$ follow similarly, noting however  that the strict transform of
$D_i$ after the first blow up acquires $A_1$ singularities, which are then resolved in the second blow up.

\smallskip

\bibliographystyle{siam}
\bibliography{bibio}

\begin{thebibliography}{10}

\bibitem{AltmannVersal}
{\sc K.~Altmann}, {\em The versal deformation of an isolated toric {G}orenstein
  singularity}, Invent. Math., 128 (1997), pp.~443--479.

\bibitem{BryanKarp2005}
{\sc J.~Bryan and D.~Karp}, {\em The closed topological vertex via the
  {C}remona transform}, J. Algebraic Geom., 14 (2005), pp.~529--542.

\bibitem{BryanKatzLeung}
{\sc J.~Bryan, S.~Katz, and N.~C. Leung}, {\em Multiple covers and the
  integrality conjecture for rational curves in {C}alabi-{Y}au threefolds}, J.
  Algebraic Geom., 10 (2001), pp.~549--568.

\bibitem{BryanPandhar2001}
{\sc J.~Bryan and R.~Pandharipande}, {\em B{PS} states of curves in
  {C}alabi-{Y}au 3-folds}, Geom. Topol., 5 (2001), pp.~287--318.

\bibitem{Burkhardt1881}
{\sc H.~Burkhardt}, {\em Untersuchungen aus dem {G}ebiete der hyperelliptischen
  {M}odulfunctionen}, Math. Ann., 38 (1891), pp.~161--224.

\bibitem{CoxMW1982}
{\sc D.~A. Cox}, {\em Mordell-{W}eil groups of elliptic curves over {${\bf
  C}(t)$} with {$p_{g}=0$} or {$1$}}, Duke Math. J., 49 (1982), pp.~677--689.

\bibitem{Cvetic:2018bni}
{\sc M.~Cveti\v{c} and L.~Lin}, {\em {TASI Lectures on Abelian and Discrete
  Symmetries in F-theory}}, PoS, TASI2017 (2018), p.~020.

\bibitem{Elkies}
{\sc N.~D. Elkies}, {\em K3 surfaces and elliptic fibrations in number theory}.
\newblock Banff Workshop 18w5190, 2018.

\bibitem{Friedman:1997yq}
{\sc R.~Friedman, J.~Morgan, and E.~Witten}, {\em {Vector bundles and F
  theory}}, Commun. Math. Phys., 187 (1997), pp.~679--743.

\bibitem{Gopakumar:1998jq}
{\sc R.~Gopakumar and C.~Vafa}, {\em M theory and topological strings. 2.},
  (1998).
\newblock ArXiv: 9812127 [hep-th].

\bibitem{GVZ2021}
{\sc A.~Grassi, A.~Verra, and A.~Zanardini}, {\em Six cusps}.
\newblock In preparation.

\bibitem{GrassiWeigand2}
{\sc A.~Grassi and T.~Weigand}, {\em On topological invariants of algebraic
  threefolds with ({$\mathbb Q$}-factorial) singularities}.
\newblock {ArXiv: 1804.02424 [math.AG]} (under revision), 2018.

\bibitem{Green:1984bx}
{\sc M.~B. Green, J.~H. Schwarz, and P.~C. West}, {\em {Anomaly Free Chiral
  Theories in Six-Dimensions}}, Nucl. Phys. B, 254 (1985), pp.~327--348.

\bibitem{Hunt96}
{\sc B.~Hunt}, {\em The geometry of some special arithmetic quotients},
  vol.~1637 of Lecture Notes in Mathematics, Springer-Verlag, Berlin, 1996.

\bibitem{KapustkaGM}
{\sc G.~Kapustka and M.~Kapustka}, {\em Fiber products of elliptic surfaces
  with section and associated {K}ummer fibrations}, Internat. J. Math., 20
  (2009), pp.~401--426.

\bibitem{Katz:1996xe}
{\sc S.~H. Katz and C.~Vafa}, {\em {Matter from geometry}}, Nucl. Phys. B, 497
  (1997), pp.~146--154.

\bibitem{KeelMatsukiMcKernanAbundance}
{\sc S.~Keel, K.~Matsuki, and J.~McKernan}, {\em Log abundance theorem for
  threefolds}, Duke Math. J., 75 (1994), pp.~99--119.

\bibitem{Kim:2019vuc}
{\sc H.-C. Kim, G.~Shiu, and C.~Vafa}, {\em {Branes and the Swampland}}, Phys.
  Rev. D, 100 (2019), p.~066006.

\bibitem{Kloosterman2007}
{\sc R.~Kloosterman}, {\em Elliptic {$K3$} surfaces with geometric
  {M}ordell-{W}eil rank 15}, Canad. Math. Bull., 50 (2007), pp.~215--226.

\bibitem{Knapp:2021vkm}
{\sc J.~Knapp, E.~Scheidegger, and T.~Schimannek}, {\em {On genus one fibered
  Calabi-Yau threefolds with 5-sections}}.
\newblock {ArXiv: 2107.05647 [hep-th]}.

\bibitem{KollarMori}
{\sc J.~Koll{\'a}r and S.~Mori}, {\em Birational geometry of algebraic
  varieties}, vol.~134 of Cambridge Tracts in Mathematics, Cambridge University
  Press, Cambridge, 1998.
\newblock With the collaboration of C. H. Clemens and A. Corti, Translated from
  the 1998 Japanese original.

\bibitem{Kuwata2000MW}
{\sc M.~Kuwata}, {\em Elliptic {$K3$} surfaces with given {M}ordell-{W}eil
  rank}, Comment. Math. Univ. St. Paul., 49 (2000), pp.~91--100.

\bibitem{Lee:2018urn}
{\sc S.-J. Lee, W.~Lerche, and T.~Weigand}, {\em {Tensionless Strings and the
  Weak Gravity Conjecture}}, JHEP, 10 (2018), p.~164.

\bibitem{Lee:2019skh}
{\sc S.-J. Lee and T.~Weigand}, {\em {Swampland Bounds on the Abelian Gauge
  Sector}}, Phys. Rev. D, 100 (2019), p.~026015.

\bibitem{Morrison:2016lix}
{\sc D.~R. Morrison, D.~S. Park, and W.~Taylor}, {\em Non-{H}iggsable abelian
  gauge symmetry and {F}-theory on fiber products of rational elliptic
  surfaces}, Adv. Theor. Math. Phys., 22 (2018), pp.~177--245.

\bibitem{Nakayama88}
{\sc N.~Nakayama}, {\em On {W}eierstrass models}, in Algebraic geometry and
  commutative algebra, {V}ol.\ {II}, Kinokuniya, Tokyo, 1988, pp.~405--431.

\bibitem{NamikawaStratified}
{\sc Y.~Namikawa}, {\em Stratified local moduli of {C}alabi-{Y}au threefolds},
  Topology, 41 (2002), pp.~1219--1237.

\bibitem{Oehlmann:2019ohh}
{\sc P.-K. Oehlmann and T.~Schimannek}, {\em {GV-Spectroscopy for F-theory on
  genus-one fibrations}}, JHEP, 09 (2020), p.~066.

\bibitem{Park:2011ji}
{\sc D.~S. Park}, {\em {Anomaly Equations and Intersection Theory}}, JHEP, 01
  (2012), p.~093.

\bibitem{RossiGeomTrans}
{\sc M.~Rossi}, {\em Geometric transitions}, J. Geom. Phys., 56 (2006),
  pp.~1940--1983.

\bibitem{RossiMIchele20174NamikawaEx}
\leavevmode\vrule height 2pt depth -1.6pt width 23pt, {\em A small and
  non-simple geometric transition}, Math. Phys. Anal. Geom., 20 (2017),
  pp.~Paper No. 15, 26.

\bibitem{SchlessingerInv1971}
{\sc M.~Schlessinger}, {\em Rigidity of quotient singularities}, Invent. Math.,
  14 (1971), pp.~17--26.

\bibitem{Schoen87}
{\sc C.~Schoen}, {\em On fiber products of rational elliptic surfaces with
  section}, Math. Z., 197 (1988), pp.~177--199.

\bibitem{ShiodaSchuettBook}
{\sc M.~Sch\"{u}tt and T.~Shioda}, {\em Mordell-{W}eil lattices}, vol.~70 of
  Ergebnisse der Mathematik und ihrer Grenzgebiete. 3. Folge. A Series of
  Modern Surveys in Mathematics [Results in Mathematics and Related Areas. 3rd
  Series. A Series of Modern Surveys in Mathematics], Springer, Singapore,
  2019.

\bibitem{ShephardTodd}
{\sc G.~C. Shephard and J.~A. Todd}, {\em Finite unitary reflection groups},
  Canad. J. Math., 6 (1954), pp.~274--304.

\bibitem{Shioda2}
{\sc T.~Shioda}, {\em Mordell-{W}eil lattices and {G}alois representation. i},
  Proc. Japan Acad., A65 (1989), pp.~268--71.

\bibitem{Taylor:2011wt}
{\sc W.~Taylor}, {\em {TASI Lectures on Supergravity and String Vacua in
  Various Dimensions}},  (2011).

\bibitem{Wazir}
{\sc R.~Wazir}, {\em Arithmetic on elliptic threefolds}, Compos. Math., 140
  (2004), pp.~567--580.

\bibitem{Weigand:2018rez}
{\sc T.~Weigand}, {\em {F-theory}}, PoS, TASI2017 (2018), p.~016.

\bibitem{Zanardini2021}
{\sc A.~Zanardini}, {\em Birational geometry of genus one fibrations and
  stability of pencils of place curves}.
\newblock Ph.D thesis, University of Pennsylvania, 2021.

\end{thebibliography}

\end{document}